\documentclass[preprint,12pt]{elsarticle1}

\usepackage{subfig}
\usepackage{graphicx}
\usepackage{caption,color}   
\usepackage{amssymb}
\usepackage{amsthm}
\usepackage{amsmath}
\usepackage{epic}
\usepackage{setspace}
\usepackage{float}
\usepackage{multirow}

\usepackage{hyperref}
\usepackage{xcolor}
\usepackage{marginnote}
\usepackage{booktabs}
\usepackage{algorithm, algorithmicx, algpseudocode}

\newtheorem{thm}{Theorem}[section]

\newtheorem{lem}[thm]{Lemma}

\newtheorem{defi}[thm]{Definition}
\newtheorem{remark}[thm]{Remark}

\numberwithin{equation}{section}
\journal{}

\begin{document}
\begin{spacing}{1.15}
\begin{frontmatter}
\title{\textbf{The subgraph eigenvector centrality of graphs}}

\author{Qingying Zhang}
\author{Lizhu Sun}\ead{lizhusun@hrbeu.edu.cn}
\author{Changjiang Bu}

\address{College of Mathematical Sciences, Harbin Engineering University, Harbin, PR China}

\begin{abstract}
Let $G$ be a connected graph and let $F$ be a connected subgraph of $G$ with a given structure.
We consider that the centrality of a vertex $i$ of $G$ is determined by the centrality of other vertices in all subgraphs contain $i$ and isomorphic to $F$.
In this paper we propose an $F$-subgraph tensor and an $F$-subgraph eigenvector centrality of $G$.
When the graph is $F$-connected, we show that the $F$-subgraph tensor is weakly irreducible, and in this case, the $F$-subgraph eigenvector centrality exists.
Specifically, when we choose $F$ to be a path $P_1$ of length $1$(or a complete graph $K_2$), the $F$-eigenvector centrality is eigenvector centrality of $G$.
Furthermore, we propose the $(K_2,F)$-subgraph eigenvector centrality of $G$ and prove it always exists when $G$ is connected.
Specifically, the $P_2$-subgraph eigenvector centrality and $(K_2,F)$-subgraph eigenvector centrality are studied.
Some examples show that the ranking of vertices under them differs from the rankings under several classic centralities.
Vertices of a regular graph have the same eigenvector centrality scores.
But the $(K_2,K_3)$-subgraph eigenvector centrality can distinguish vertices in a given regular graph.
\end{abstract}


\end{frontmatter}

\section{Introduction}
Centrality is a measure of the importance of vertices in complex networks according to a certain viewpoint.
Over the years, researchers have proposed many centrality measures \cite{albert2000error,freeman1991centrality,alahakoon2011k,estrada2005subgraph,zhou2023estrada,bugedo2024family}.
Eigenvector centrality is an important centrality measure \cite{bonacich1972factoring}.
It is widely applied in social networks \cite{bihari2015eigenvector},
brain science networks \cite{lohmann2010eigenvector},
protein networks \cite{negre2018eigenvector},
hypernetworks \cite{benson2019three,tudisco2018node},
multilayer networks \cite{wang2018new},
 Google's PageRank search engine \cite{langville2005survey},
 and so on.

Let $A = (a_{ij})_{n \times n}$ be  the adjacency matrix of graph $G$ with $n$ vertices,
where
\begin{align*}
a_{ij}=
\begin{cases}
1,&  if \text{ $ i \sim j $,} \\
0,&   \text{ otherwise}.
\end{cases}
\end{align*}

\begin{defi} \cite{bonacich1972factoring}  \label{ec}
For a connected graph $G$ with $n$ vertices,
let $A$ be the adjacency matrix of $G$, and $\rho(A)$ be the spectral radius of $G$.
The positive solution $\boldsymbol {x}=(x_1,x_2,{\cdots},x_n)^{\mathrm{T}}$ to the equation $A \boldsymbol {x} = \rho(A) \boldsymbol {x}$ is the eigenvector centrality of $G$.
\end{defi}

By the Perron-Frobenius theorem of nonnegative matrices \cite{meyer2023matrix}, the positive eigenvector $\boldsymbol {x}$ corresponding to the spectral radius of a connected graph $G$ is unique (up to its scalar multiples).
Let $[n]=\{1,2,\cdots,n\}$.
For vertex $i \in [n]$,
from Definition \ref{ec}, we know that
\begin{align}\label{90999}
x_i = \frac{1}{ \rho(A)} \sum_{j=1}^n a_{ij} x_j = \frac{1}{ \rho(A)}  \sum_{  \{i,j\} \in {E(G)}} x_j.
\end{align}
That is the centrality of $i$ is determined by the centrality of other vertices that are in edges contain $i$.
An edge of $G$ can be regarded as a subgraph of $G$, that is a path $P_1$ of length $1$(or a complete graph $K_2$ ).
So the eigenvector centrality of $i$ is determined by the eigenvector centrality of other vertices in subgraphs $P_1$(or $K_2$) that are contain $i$.
Inspired by the eigenvector centrality of graphs, we propose an $F$-subgraph eigenvector centrality of graphs.
That is the centrality of a vertex $i$ of $G$ to be determined by the centrality of other vertices in all subgraphs contain $i$ and isomorphic to $F$.
We propose an $F$-subgraph tensor of graphs.


\begin{defi} \label{ztzldy}
For a connected graph $G$ with $n$ vertices,
let $F$ be a connected subgraph of $G$ with $k$ vertices.
We use $V_F$ to denote a set consisting of vertex sets of subgraphs isomorphic to $F$.
Let $\mathcal{F}(v_1,  \cdots , v_k)$ be a set of subgraphs isomorphic to $F$ with the vertices set  $\{v_1,  \cdots , v_k\}$.
Let
\begin{align*}
a_{i_1 i_2 \cdots  i_k} =
\begin{cases}
|\mathcal{F}(i_1, i_2, \cdots , i_k)| ,&  if \text{ $  \{ i_1, i_2, \cdots , i_k\} \in V_F $,} \\
0,&    \text{ otherwise},
\end{cases}
\end{align*}
where $|\mathcal{F}(i_1,  \cdots , i_k)|$ denotes the number of elements in the set $\mathcal{F}(i_1,  \cdots , i_k)$.
The $k$-th order $n$-dimensional tensor $\mathcal{A}_F = (a_{i_1i_2 \cdots i_k})$ is called the $F$-subgraph tensor of $G$.
\end{defi}

Let $\rho(\mathcal{A}_F)$ be the spectral radius of $\mathcal{A}_F $, and let
\begin{align*}
 x_{i}^{k-1} = \frac{1}{\rho(\mathcal{A}_F)}
               \sum_{\{ i,i_2,\cdots,i_k  \} \in V_F  }|\mathcal{F}(i,i_2, \cdots, i_k)|   x_{i_2} \cdots x_{i_k},
\end{align*}
where $x_i$ is the centrality score of $i$, $i \in [n]$.
Then
\begin{align}\label{1}
\mathcal{A}_F \boldsymbol {x} ^{k-1}= \rho(\mathcal{A}_F) \boldsymbol {x} ^{[k-1]},
\end{align}
where $\boldsymbol {x}^{[k-1]}=(x_1^{k-1},x_2^{k-1},{\cdots},x_n^{k-1})^{\mathrm{T}}$.

When the subgraph tensor $\mathcal{A}_F$ is weakly irreducible, from Lemma \ref{perron} we can know that the equation (\ref{1}) has a unique positive solution $\boldsymbol {x}$ (up to its scalar multiples).
Specifically, when $F$ is chosen to be $P_1$(or $K_2$),
$\mathcal{A}_F =(a_{ij})$ is the adjacency matrix of $G$, $\boldsymbol {x}$ is the eigenvector centrality of graphs.

\begin{defi} \label{zitudingyi}
Let $G$ be a connected graph.
For a connected subgraph $F$ of $G$ with $k$ vertices,
$\mathcal{A}_F$ is the subgraph tensor.
If $\mathcal{A}_F$ is weakly irreducible,
the positive solution $\boldsymbol {x}$ of $\mathcal{A}_F \boldsymbol {x} ^{k-1}= \rho(\mathcal{A}_F) \boldsymbol {x} ^{[k-1]}$  is called the $F$-subgraph eigenvector centrality of $G$.
\end{defi}

There has been some research on  special subgraph tensors for studying graph structure and centrality.
The  $2$-star subgraph tensor and it's eigenvalue problems have  already  been proposed in \cite{petersdorf1969spektrum,cvetkovic1980spectra}.
The literature \cite{liu2023high} and \cite{liu2023generalization} respectively proposed the $r$-star tensor and the $r$-clique tensor of graphs and used their spectra  to study the extremal problems of graphs.
The literature \cite{xu2023two}  proposed the  2-step tensor and the 2-step eigenvector centrality of graphs.
In this paper, the general subgraph tensor and the general subgraph eigenvector centrality of graphs are proposed (that is the Definition \ref{ztzldy}, \ref{zitudingyi}).
Thus the definitions of subgraph tensor  in \cite{cvetkovic1980spectra,liu2023high,liu2023generalization,xu2023two} are  special cases of Definition \ref{ztzldy}.

In this paper we consider that the centrality of a vertex $i$ of $G$ is determined by the centrality of other vertices in all subgraphs  contain $i$ and isomorphic to $F$.
We propose an $F$-subgraph tensor and an $F$-subgraph eigenvector centrality of $G$.
Furthermore, we propose the $(K_2,F)$-eigenvector centrality of $G$.
We give a necessary and sufficient condition for the existence of $F$-subgraph eigenvector centrality and prove the $(K_2,F)$-subgraph eigenvector centrality always exists.
The organization of this paper is as follows.
In Section 2, we introduce some concepts and lemmas related to tensors.
In Section 3, we study the $P_2$-subgraph eigenvector centrality and the $(K_2,F)$-subgraph eigenvector centrality,
and give the  existence of the two centrality measures.
The degree centrality and eigenvector centrality of the vertices are always the same in regular graphs.
We provide an example showcasing that the ($K_2,K_3$)-subgraph eigenvector centrality can distinguish vertices in a regular graph.
In Section 4, the centrality measures proposed in this paper are applied to some real-world networks, and the results show that the vertices with high centrality scores are generally in more given subgraphs.
However, the centrality score of a vertex is not entirely determined by the number of given subgraphs contain the vertex, the centrality measures proposed in this paper are  global centrality measures.

\section{Preliminaries }
Let $\mathbb{C}^n$ and $\mathbb{C}^{[k,n]}$ denote the sets of $n$-dimensional vectors and $k$-th order $n$-dimensional tensors over the complex number field $\mathbb{C}$, respectively.
A tensor $\mathcal{A}=(a_{{i_1} {i_2} {\cdots}  {i_k}}) \in \mathbb{C}^{[k,n]} $ is a multi-dimensional array containing $n^k$ elements, where $i_j \in [n], j \in [k]$.
For $\boldsymbol {x}=(x_1,x_2,{\cdots},x_n)^{\mathrm{T}} \in \mathbb{C}^n$,
$\mathcal{A}{\boldsymbol {x}}^{k-1}$ is a vector in $\mathbb{C}^n$,
with the $i$-th component
\begin{align*}
(\mathcal{A}{\boldsymbol {x}}^{k-1})_i=\sum_{{i_2}, {\cdots},  {i_k}=1}^{n}  a_{{i} {i_2} {\cdots} {i_k}} {x_{i_2}}{x_{i_3}}{\cdots}{x_{i_k}}, i \in [n].
\end{align*}
If there exists  $\lambda \in \mathbb{C}$ and a nonzero vector $\boldsymbol {x}$ such that
\begin{align*}
\mathcal{A}{\boldsymbol {x}}^{k-1}=\lambda \boldsymbol {x}^{[k-1]},
\end{align*}
then $\lambda$  is called an eigenvalue of $\mathcal{A}$,
and $\boldsymbol {x}$ is an eigenvector of  $\mathcal{A}$ corresponding to $\lambda$ \cite{qi2005eigenvalues,lim2005singular},
where $\boldsymbol {x}^{[k-1]}=(x_1^{k-1}, {\cdots}, x_n^{k-1})^{\mathrm{T}}$.
The eigenvalues of tensors have attracted extensive attention since they have been proposed \cite{qi2017tensor,cooper2012spectra,sun2016moore,chen2024spectra}.
And they have  been applied to hypergraph clustering \cite{chang2020hypergraph}, image fusion \cite{sun2023nf},
crystallography \cite{chen2023c}, mechanics \cite{nikabadze2016eigenvalue}, and so on.

A tensor is nonnegative if  all its elements are nonnegative.
For a  nonnegative tensor $\mathcal{A}=(a_{{i_1} {i_2} {\cdots}  {i_k}})$, ${i_j} \in [n]$, $j \in [k]$,
let $D_{\mathcal{A}}=(V{(D_{\mathcal{A}})},E{(D_{\mathcal{A}})})$ be the associated directed graph of $\mathcal{A}$,
with the vertex set $V{(D_{\mathcal{A}})}=[n]$, and the arc set
$E{(D_{\mathcal{A}})}=\left\{ {(i,j)| a_{i {i_2}{\cdots}{i_k} } \neq 0 ,  j \in {\left\{{{i_2},{\cdots},{i_k}}\right\}}  }\right\}$ (\cite{qi2017tensor}).
For any distinct $i,j \in {V(D_{\mathcal A})}$,
if there exists a directed path from  $i$ to $j$ and $j$ to $i$,
then $D_{\mathcal{A}}$  is said to be strongly connected.

The relationship between the weak irreducibility of nonnegative tensors and their associated directed graphs, as well as some results of the Perron-Frobenius theorem for nonnegative tensors have been given (see \cite{lim2005singular,chang2008perron,yang2010further,friedland2013perron}).

\begin{lem} \cite{friedland2013perron} \label{perron}
The nonnegative tensor $\mathcal{A}$ is weakly irreducible if and only if $D_{\mathcal{A}}$ is strongly connected.
If $\mathcal{A}$ is a nonnegative weakly irreducible tensor, then the spectral radius $\rho (\mathcal{A})$ is an eigenvalue of $\mathcal{A}$, and there exists a unique positive eigenvector corresponding to $\rho (\mathcal{A})$ (up to its scalar multiples).
\end{lem}

From Lemma \ref{perron}, when the subgraph tensor $\mathcal{A}_F$ is weakly irreducible, the equation (\ref{1}) has a positive solution $\boldsymbol {x}$. That implies the $F$-subgraph eigenvector centrality of graphs exists.
Next, some definitions and a necessary and sufficient condition for the weak irreducibility of $F$-subgraph tensor are given.

\begin{defi} \label{ztltddy}
For a connected graph $G$,
let $F$ be a connected subgraph of $G$.
Let $P$ be a path in $G$.
If every edge on $P$ is in some subgraph of $G$ that is isomorphic to $F$, then $P$ is called an $F$-path of $G$.
If there exists an $F$-path between any two distinct vertices in $G$, then $G$ is said to be $F$-connected.
\end{defi}

\begin{thm} \label{liantong}
Let $G$ be a graph.
For a connected subgraph $F$ of $G$ with $k$ vertices (where $k \geq 2$), the subgraph tensor $\mathcal{A}_F$  is weakly irreducible if and only if $G$ is $F$-connected.
\end{thm}

\begin{proof}
Let $D_{\mathcal{A}_F}$ be the directed graph associated with $\mathcal{A}_F$.

When $G$ is $F$-connected. By Definition \ref{ztltddy}, for $i, j \in V(G)$, $i \neq j$,
there exists an $F$-path $P_s = i_1 e_1 i_2 e_2 \cdots i_s e_s i_{s+1}$ in $G$,
where $i=i_1, j=i_{s+1}$.
For each edge $e_t = \{i_t, i_{t+1}\}$ on $ P_s $, there exists a subgraph isomorphic to $F$ with vertex set
$ \{ i_t, i_{t+1}, u_3^t, \cdots, u_k^t  \}$,  hence $a_{i_t i_{t+1}  u_3^t  \cdots u_k^t } \neq 0$,
where $t \in [s]$.
Therefore, there exists a directed arc from $i_t$ to $i_{t+1}$ in $D_{\mathcal{A}_F}$, which means that there is a directed path from $i$ to $j$ in $D_{\mathcal{A}_F}$.
So $D_{\mathcal{A}_F}$ is  strongly connected.
By Lemma \ref{perron}, we can know that $\mathcal{A}_F$  is weakly irreducible.

When $\mathcal{A}_F$ is weakly irreducible. For $i, j \in V(G)$, $i \neq j$,  there exists a directed path $ \widetilde{P}_s =  (i_1,i_2)(i_2,i_3) \cdots (i_s,i_{s+1})$ in $D_{\mathcal{A}_F}$,  where $i_1=i, i_{s+1}=j$.
From the definition of $D_{\mathcal{A}_F}$ we have $a_{i_t  i_{t+1}  u_3^t  \cdots  u_k^t } \neq 0$ in $\mathcal{A}_F$, where $t \in [s]$, $ u_3^t,  \cdots , u_k^t \in V(G)$.
That is, there exists a subgraph $F_t$ isomorphic to $F$ in $G$, and the vertex set of $F_t$ is $ \{ i_t, i_{t+1}, u_3^t, \cdots, u_k^t  \}$.
Therefore $i_1 e_1 i_2 e_2 \cdots i_s e_s i_{s+1}$ is an $F$-path.
So $G$ is $F$-connected.
\end{proof}


\section{Subgraph eigenvector centralities}
In this section, we chose $F$ to be a path $P_2$ of length $2$, and give a necessary and sufficient condition for the weak irreducibility of the $P_2$-subgraph tensor, and analyze the $P_2$-subgraph eigenvector centrality.
In numerical example, the vertices with high rankings under $P_2$-subgraph eigenvector centrality are also in more subgraph $P_2$.

Furthermore, we propose the ($K_2,F$)-subgraph eigenvector centrality and prove it always exists.
When we chose $F$ to be $K_3$, the ($K_2, K_3$)-subgraph eigenvector centrality can distinguish vertices in a regular graph.

\subsection{The $P_2$-eigenvector centrality of graphs}
In network analysis, path is a relatively important subgraph.
Some classic centrality measures are based on paths \cite{bonacich1972factoring,albert2000error,freeman1991centrality,alahakoon2011k}, and paths can be used to study the transmission of information in graphs and analyze the structure of graphs \cite{qi2017eb,de2022communication}.
In this section, we chose $F$ to be $P_2$.
For a connected graph $G$, we prove that the $P_2$-subgraph tensor of $G$ is weakly irreducible when $|E(G)| \geq 2$.

From Definition \ref{ztzldy} we have for a connected graph $G$ with $n$ vertices,
the $P_2$-subgraph tensor of $G$ is denoted by $\mathcal{A}_{P_2}=(a_{ijk})$,
where
\begin{align*}
a_{ijk}=
\begin{cases}
|\mathcal{P}_2 (i, j, k)| , &  if \text{ ${\{i, j, k \} \in V_{P_2}  } $,} \\
0,&   \text{ otherwise},
\end{cases}
\end{align*}
and $\mathcal{P}_2 (i, j, k)$ is a set of subgraphs  isomorphic to $P_2$ with the vertices set  $\{i,j,k\}$.


\begin{thm}
The $P_2$-subgraph tensor $\mathcal{A}_{P_2}$ of a connected graph $G$ is weakly irreducible if and only if $|E(G)| \geq 2$.
\end{thm}

\begin{proof}
When $|E(G)| \geq 2$.
Because $G$ is connected, every edge in $G$ is in a path of length $2$.
It follows that $G$ is $P_2$-connected. By Theorem \ref{liantong}, $\mathcal{A}_{P_2}$ is weakly irreducible.

When $\mathcal{A}_{P_2}$ is weakly irreducible.
By Theorem \ref{liantong}, $G$ is $P_2$-connected. Then $|E(G)| \geq 2$.
\end{proof}

Next we calculate the $P_2$-subgraph eigenvector centrality in an example and compare it with eigenvector centrality and degree centrality of graphs.
We denote $P_2$-subgraph eigenvector centrality, degree centrality and eigenvector centrality by $P_2C$, $DC$ and $EC$, respectively.
For a vertex $i$ of a graph $G$, we denote the number of $P_2$ contain $i$ by $NP_2(i)$.
The graph in Figure \ref{fig1} is a connected graph with $10$ vertices. We use Algorithm \ref{A1} in Section \ref{sec4} to calculate $P_2C$ for the vertices of the graph in Figure \ref{fig1}, and compare $P_2C$ with $DC$, $EC$ and $NP_2(i)$, as shown in Table \ref{table1}.

\begin{figure}[H]
\centerline{\includegraphics[scale=0.2]{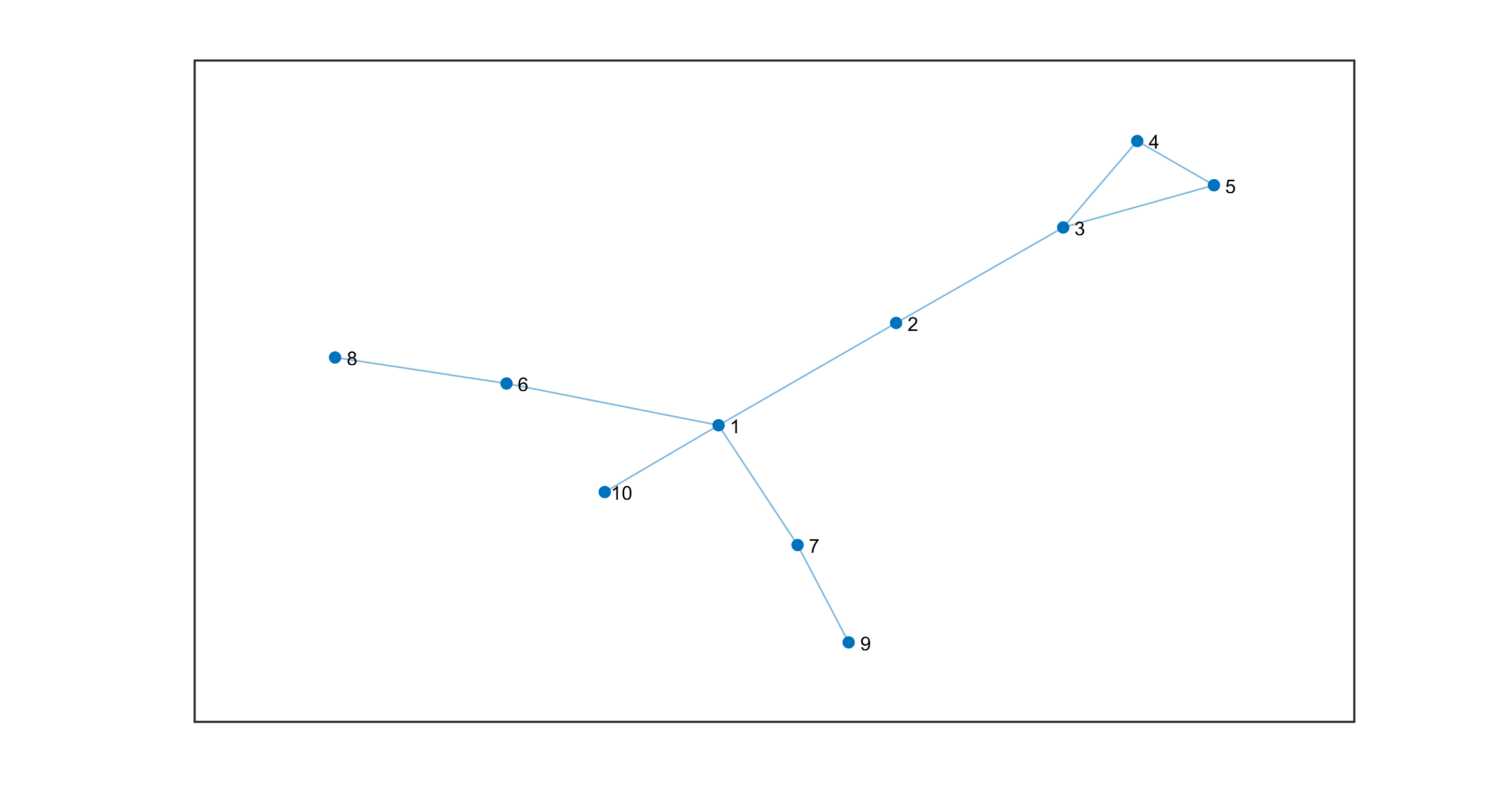}}
\caption{A graph with 10 vertices.}
\label{fig1}
\end{figure}

\begin{table}[htbp]
    \small
    \centering
    \caption{The centralities of the vertices in Fig1.}
    \label{table1}
    \begin{tabular}{l*{10}{c}}
    \toprule
        \textbf{Measures} & \textbf{1} & \textbf{2} & \textbf{3} & \textbf{4,5} & \textbf{6,7} & \textbf{8,9} & \textbf{10} \\ \midrule
        $DC$           & \textbf{4} & 2 & 3 & 2 & 2 & 1 & 1   \\
        $NP_2(i)$       & \textbf{9} & 6  & 6 & 4 & 4 & 1 & 3    \\
        $EC$  & 0.4567 & 0.3051 & \textbf{0.4674} & 0.3491 & 0.2389 & 0.1022 & 0.1953  \\
        $P_2C$   & \textbf{0.4180} &0.3929 &0.3774 & 0.3175 & 0.3093 & 0.1603 & 0.2901  \\
    \bottomrule
    \end{tabular}
\end{table}

From Table \ref{table1} we can know that the ranking of vertices under $P_2C$ is different from $DC$ and $EC$.
Vertex $1$  has the highest ranking in $P_2C$, and comparing with other vertices $NP_2(1)$ is the largest.
We can see that $NP_2(2)=NP_2(3)$,
but vertex $2$ has a higher score than vertex $3$ under $P_2C$.

\begin{remark}
The centrality score of vertex $i$ under $P_2C$ is not entirely determined by $NP_2(i)$, because $P_2C$ is a global centrality measure.
\end{remark}

\subsection{The ($K_2,F$)-subgraph eigenvector centrality of  graphs}
When $F$ is a general subgraph, the $F$-subgraph tensor may be reducible, in which case the $F$-subgraph eigenvector centrality does not exist.
We consider selecting both $K_2$ and $F$ in a graph $G$ to form a new subgraph tensor $\mathcal{A}_{K_2, F}$.
In the subsequent text, we prove that the ($K_2,F$)-subgraph tensor $\mathcal{A}_{K_2, F}$ of $G$ is weakly irreducible when $G$ is connected.
Thus when $G$ is connected, the ($K_2,F$)-subgraph eigenvector centrality always exists.

From Definition \ref{zitudingyi},
the $F$-subgraph eigenvector centrality score for a vertex $i \in V(G)$ is denoted by $x_i$,
\begin{align}    \label{88999}
\rho(\mathcal{A}_F)  x_{i}^{k-1} =
\sum_{\{ i,i_2,\cdots,i_k  \} \in V_F  }|\mathcal{F}(i,i_2, \cdots, i_k)|   x_{i_2} \cdots x_{i_k}.
\end{align}

When we choose $F$ to be $K_2$, the $K_2$-subgraph eigenvector centrality score for a vertex $i \in V(G)$  is denoted by $x_i$, we have
\begin{align} \label{8899}
\rho(\mathcal{A}_{K_2}) x_i = \sum_{i \sim j} x_j .
\end{align}


By multiplying both sides of Equation (\ref{8899}) by $x_i^{k-2}$, we obtain that
\begin{align} \label{889999}
\rho(\mathcal{A}_{K_2}) x_i^{k-1} = \sum_{ i \sim j}  x_j  x_i^{k-2}.
\end{align}

Let
\begin{align} \label{8800000}
\rho(\mathcal{A}_{K_2, F}) x_i^{k-1} =
\sum_{i \sim j}  x_j  x_i^{k-2}+
\sum_{\{ i,i_2,\cdots,i_k  \} \in V_F  }|\mathcal{F}(i,i_2, \cdots, i_k)|   x_{i_2} \cdots x_{i_k}  ,
\end{align}
where $\mathcal{A}_{K_2, F}=(a_{i i_2 \cdots i_k })$ is called the ($K_2,F$)-subgraph tensor of $G$, and
\begin{align} \label{7676}
a_{i i_2 \cdots i_k }=
\begin{cases}
1 ,                                &  if \text{ $ i_2 \in N(i), i_3, \cdots, i_k = i  $},\\
|\mathcal{F} (i,i_2,\cdots,i_k)|  ,&  if \text{ $ \{i,i_2,\cdots,i_k  \} \in V_F  $,}\\
0 ,                                &  otherwise,
\end{cases}
\end{align}
where $N(i)$ represents the neighbor set of $i$.
Then we have
\begin{align} \label{0101}
\mathcal{A}_{K_2, F} \boldsymbol {x}^{k-1} =\rho(\mathcal{A}_{K_2, F}) {\boldsymbol {x}}^{[k-1]} ,
\end{align}
where $\boldsymbol {x} = (x_1, \cdots, x_n  )^\mathrm{T}$.

\begin{thm} \label{F}
For a connected graph $G$,
let $F$ be a connected subgraph of $G$.
Then the subgraph tensor $\mathcal{A}_{K_2, F}$ is weakly irreducible.
\end{thm}

\begin{proof}
Let $D_{\mathcal{A}_{K_2, F}}$ be the associated directed graph of $\mathcal{A}_{K_2, F}$.
Because $G$ is connected, there exists a path $P_s = i_1 e_1 i_2 e_2 \cdots  i_s e_s i_{s+1}$ between any two distinct vertices $i, j \in V(G) $,
where $i=i_1$ and $j=i_{s+1}$.
Obviously, in $\mathcal{A}_{K_2, F}$ we have $a_{i_t i_{t+1} i_t  \cdots i_t } \neq 0$, for $t \in [s]$.
Then, in $D_{\mathcal{A}_{K_2, F}}$ there exists a directed arc from $i_t$ to $i_{t+1}$, which means that there is a directed path from $i$ to $j$.
So $D_{\mathcal{A}_{K_2, F}}$ is strongly connected.
By Lemma \ref{perron}, we can know that $\mathcal{A}_{K_2, F}$  is weakly irreducible.
\end{proof}

By Theorem \ref{F}, $\mathcal{A}_{K_2, F}$ is weakly irreducible when $G$ is connected.
So the equation $\mathcal{A}_{K_2, F} \boldsymbol {x} ^{k-1}=
 \rho(\mathcal{A}_{K_2, F}) \boldsymbol {x} ^{[k-1]}$ has a unique positive solution $\boldsymbol {x}$ (up to its scalar multiples).
\begin{defi}
For a connected graph $G$,
let $F$ be a connected subgraph of $G$ with $k$ vertices.
The positive solution $\boldsymbol {x}$ of $\mathcal{A}_{K_2, F} \boldsymbol {x} ^{k-1}=
 \rho(\mathcal{A}_{K_2, F}) \boldsymbol {x} ^{[k-1]}$ is called the ($K_2,F$)--subgraph eigenvector centrality of $G$.
\end{defi}


The complete graph $K_3$  have significant importance for studying social networks and understanding community structures\cite{rosvall2014memory,benson2015tensor,zhang2020community}.
Next we chose $F$ to be $K_3$ and study the ($K_2,K_3$)-subgraph eigenvector centrality.

Let $\mathcal{A}_{K_2, K_3} = (a_{ijk})$ be the ($K_2,K_3$)-subgraph tensor of  $G$ for $K_2$ and $K_3$, from (\ref{7676}) we have
\begin{align} \label{00}
a_{ijk}=
\begin{cases}
1 ,                         &  if \text{ $ j  \in N(i), k = i  $},\\
|\mathcal{K}_3 (i, j, k)|  ,&  if \text{ $ \{i, j, k  \} \in V_{K_3}  $,}\\
0 ,&   otherwise ,
\end{cases}
\end{align}
and $\mathcal{K}_3 (i, j, k)$ is a set of some subgraphs  isomorphic to $K_3$ with the vertices set  $\{i,j,k\}$.

We denote ($K_2,K_3$)-subgraph eigenvector centrality by $(K_2,K_3)C$.
Next, we calculate $(K_2,K_3)C$ of a graph in an example and analyze the characteristics of vertices that ranking high in $(K_2,K_3)C$.
The graph in Figure \ref{fig6} is a regular graph with $8$ vertices, each of degree $3$ \cite{estrada2005subgraph}.
In Table \ref{table2}, we list the scores of the vertices in terms of $DC$, $EC$ and $(K_2,K_3)C$.

From Table \ref{table2}, it can be observed that the scores of the vertices under $DC$ and $EC$ are the same.
Under $(K_2,K_3)C$, vertices $1$, $2$, and $8$ have the highest scores.
We can observe that vertices $1$, $2$, and $8$ are in a subgraph $K_3$ of the regular graph.
The scores of vertices $3$, $5$, and $7$ under $(K_2,K_3)C$ are higher than vertices $4$ and $6$,
we also observe that vertices $3$, $5$, and $7$ are respectively adjacent to vertices $2$, $1$, and $8$ of $K_3$.

\begin{figure}[H]
\centerline{\includegraphics[scale=0.23]{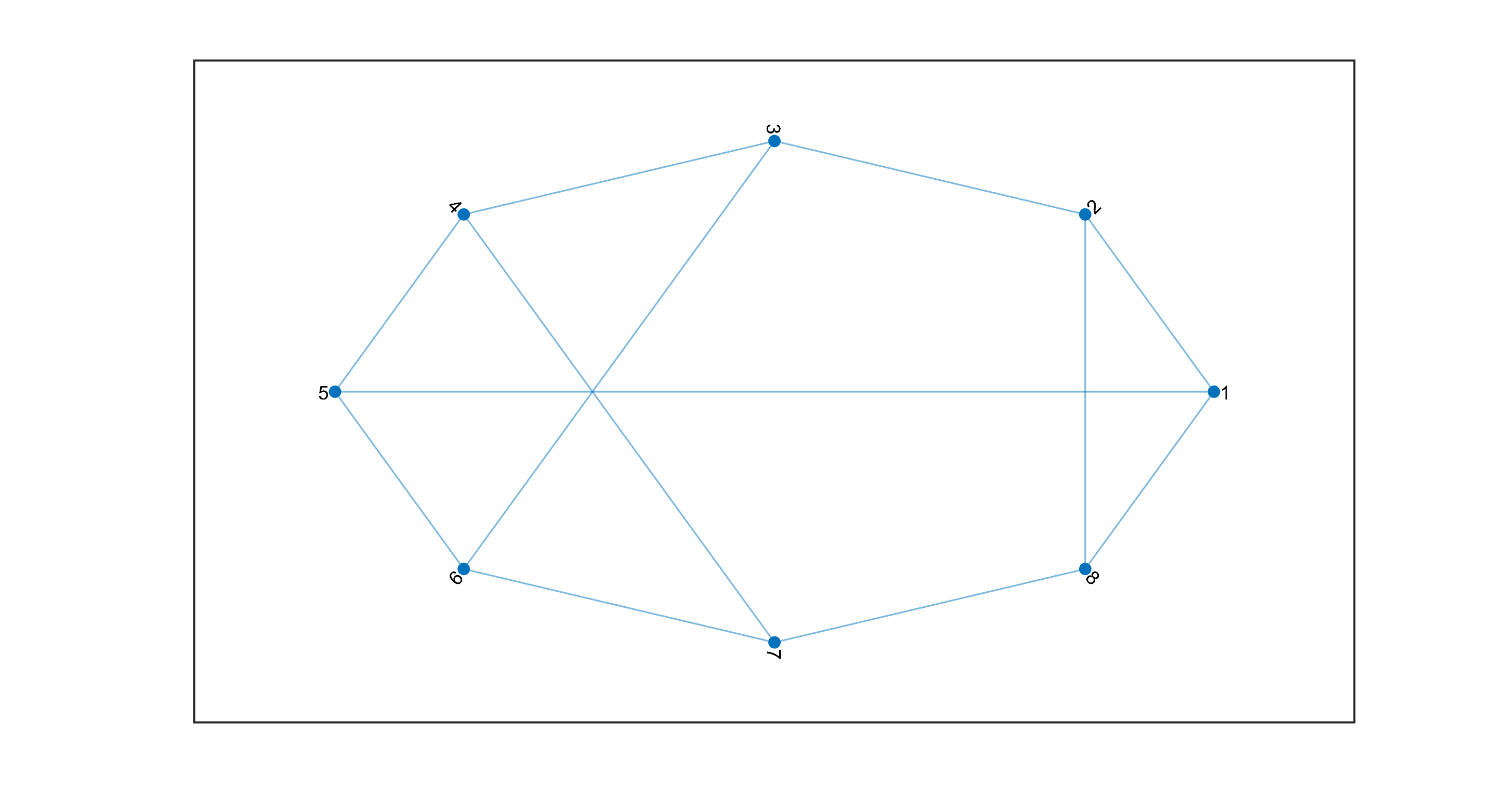}}
\caption{A regular graph.}
\label{fig6}
\end{figure}

\begin{table}[htbp]
    \small
    \centering
    \caption{The centralities of the nodes in Fig2.}
    \label{table2}
    \begin{tabular}{l*{10}{c}}
    \toprule
        \textbf{Measures} & \textbf{1,2,8} & \textbf{3,5,7} &  \textbf{4,6} &  \\ \midrule
        $DC$  &3  &3  &3 \\
        $EC$  &0.3536  &0.3536  &0.3536 \\
        $(K_2,K_3)C$  & \textbf{0.5381} & 0.1821 & 0.1259 \\
    \bottomrule
    \end{tabular}
\end{table}

\section{Numerical example} \label{sec4}
In \cite{zhou2013efficient}, researchers proposed an algorithm(ZQW) for computing the spectral radius and corresponding eigenvector of a nonnegative weakly irreducible tensor.
This paper uses the ZQW-algorithm to calculate the subgraph eigenvector centrality of graphs,
the algorithm is as follows.

\begin{algorithm}[h]
    \small
    \renewcommand{\algorithmicrequire}{\textbf{Input:}}
    \renewcommand{\algorithmicensure}{\textbf{Output:}}
    \caption{The subgraph eigenvector centrality of graphs}
    \label{A1}
    \begin{algorithmic}[1]
        \Require A graph $G$
        \Ensure The subgraph eigenvector centrality $\boldsymbol {x}$

        \State Form the subgraph tensor $\mathcal{A}_{F}$ by finding all subgraphs in $G$ that are isomorphic to $F$.
        \State Choose an $n$-dimension vector ${\boldsymbol {x}}^{(0)} > 0$.
 Let $\mathcal{B}_{F} = \mathcal{A}_{F} + \mathcal{I}$ and $\boldsymbol {y}^{(0)}
 =\mathcal{B}_{F}{(\boldsymbol {x}^{(0)})^2}$.
 Set $k=1$.
        \State Compute
        \Statex   $\boldsymbol {x}^{(k)}= \frac{{\boldsymbol {y}^{(k-1)}}^{[\frac{1}{2}]}   }
        { ||{\boldsymbol {y}^{(k-1)}}^{[\frac{1}{2}]}||} ,
        {\boldsymbol {y}}^{(k)}={\mathcal{B}}_{F}{({\boldsymbol {x}}^{(k)})^2}$
        \Statex  $\underline{\lambda}^k =\min_{x_i^{(k)}> 0 } \frac{y_i^{(k)}}{(x_i^{(k)})^2}  $ ,
                     $\overline{\lambda}_k = \max_{x_i^{(k)}> 0 } \frac{y_i^{(k)}}{(x_i^{(k)})^2}$.
        \State If $\underline{\lambda}^k = \overline{\lambda}_k$,
        the algorithm stops and then $\boldsymbol {x}^{(k)}$ is the subgraph tensor eigenvector centrality;
        Otherwise, replace $k$ by $k+1$ and go to step 3.
 \end{algorithmic}
\end{algorithm}

In this section, we calculate the $P_2$-subgraph eigenvector centrality and the ($K_2, K_3$)-subgraph eigenvector centrality in two real-world networks. Moreover we compare and analyze them with eigenvector centrality, betweenness centrality\cite{freeman1991centrality} and subgraph centrality\cite{estrada2005subgraph}.
We denote betweenness centrality by $BC$,  the expression for $BC$ of vertex $v$ is $C(v)= \sum_{s, v, t} \frac{\sigma_{st}(v)}{\sigma_{st}}$,
where $s, v, t$ are vertices that are distinct from each other,
$\sigma_{st}$ is the total number of shortest paths from vertex $s$ to vertex $t$, and $\sigma_{st}(v)$ is the number of shortest paths that pass through vertex $v$.
The subgraph centrality is denoted by $SC$, the expression for $SC$ of vertex $u$ is
$C(u)= \left(\exp(A)\right)_{u u}= \sum_{k=0}^{\infty} \frac{\left(A^k\right)_{u u}}{k !}$,
where $A$ is the adjacency matrix of graph $G$.

\subsection{The Sandi-Auths network}
The Sandi-Auths network is a social network consisting of $86$ vertices and $124$ edges. The vertices in the network represent political candidates in San Diego, California, and the edges represent the mutual support relationships between these candidates \cite{walteros2019detecting}.
In \cite{bugedo2024family}, a centrality measure based on subgraphs counting was proposed.
Specifically, the All-Trees centrality($TS$) has been studied when the subgraphs are trees,
and experiment was conducted in Sandi-Auths network.
We also use several centrality measures to rank the vertices in Sandi-Auths network, including $EC$, $BC$, $SC$, $P_2C$, $(K_2,K_3)C$ and $TS$.
In Table \ref{table4}, we list the top five vertices ranked by the aforementioned six centrality measures. In Figure \ref{fig11}, we present the heatmap of the Sandi-Auths network under these six centrality measures, and highlight the top five vertices in the graph.

\begin{table*}[h]
    \small
    \centering
    \caption{The top five vertices in Figure 3 under six centrality measures.}
    \label{table4}
    \begin{tabular}{lccccccccc}  
    \toprule
        \textbf{Measures} & \textbf{rank1} & \textbf{rank2} & \textbf{rank3} & \textbf{rank4} & \textbf{rank5}  \\ \midrule
        $EC$        & 33 & 36 & 2 &44 & 26  \\
        $P_{2}C$   & 36 & 33 & 30 &44 & 3  \\ 
        $(K_2,K_3)C$   & 2 & 66 & 18 & 32 & 76   \\
        $BC$        & 33 & 30 & 36 & 58 & 3  \\
        $SC$        & 30 & 36 & 33 & 3 & 2 \\
        $TS$        &36 &30 &40 &33 &2 \\
    \bottomrule
    \end{tabular}
\end{table*}

\begin{figure}
\centering
\subfloat[ $EC$]{
\includegraphics[width=0.35\textwidth]{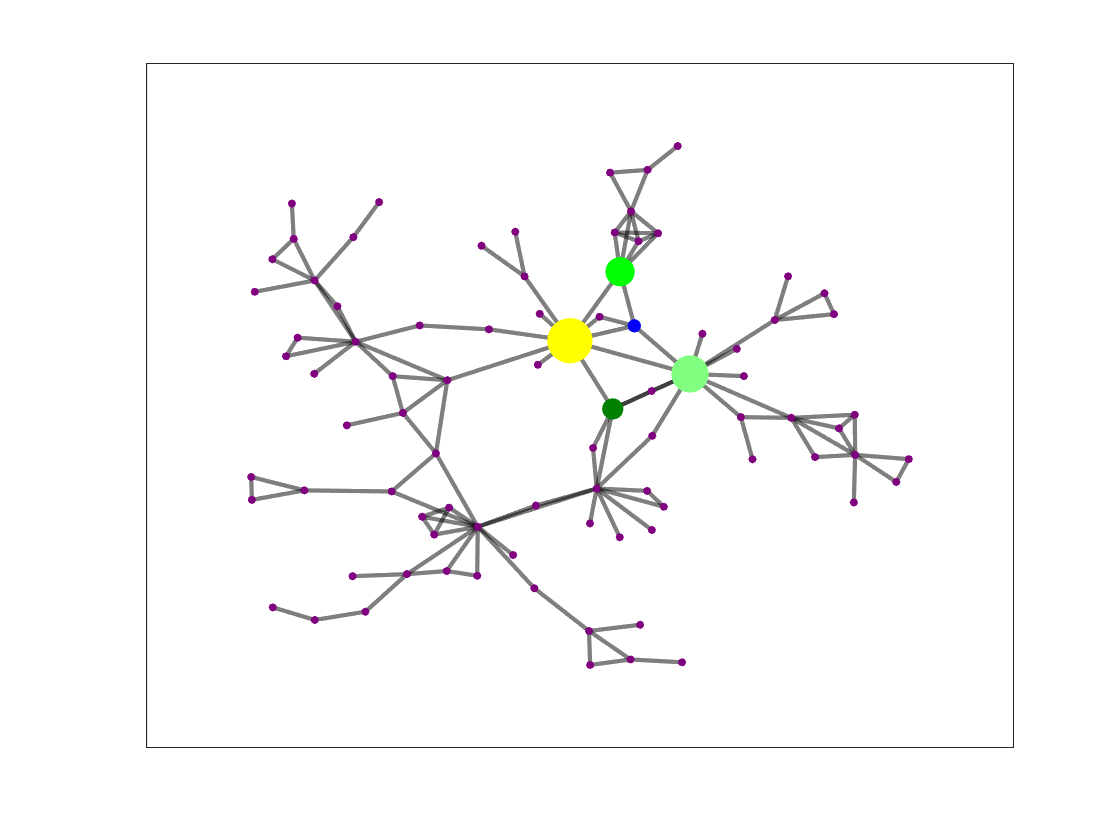}
}
\subfloat[ $P_2C$]{
\includegraphics[width=0.35\textwidth]{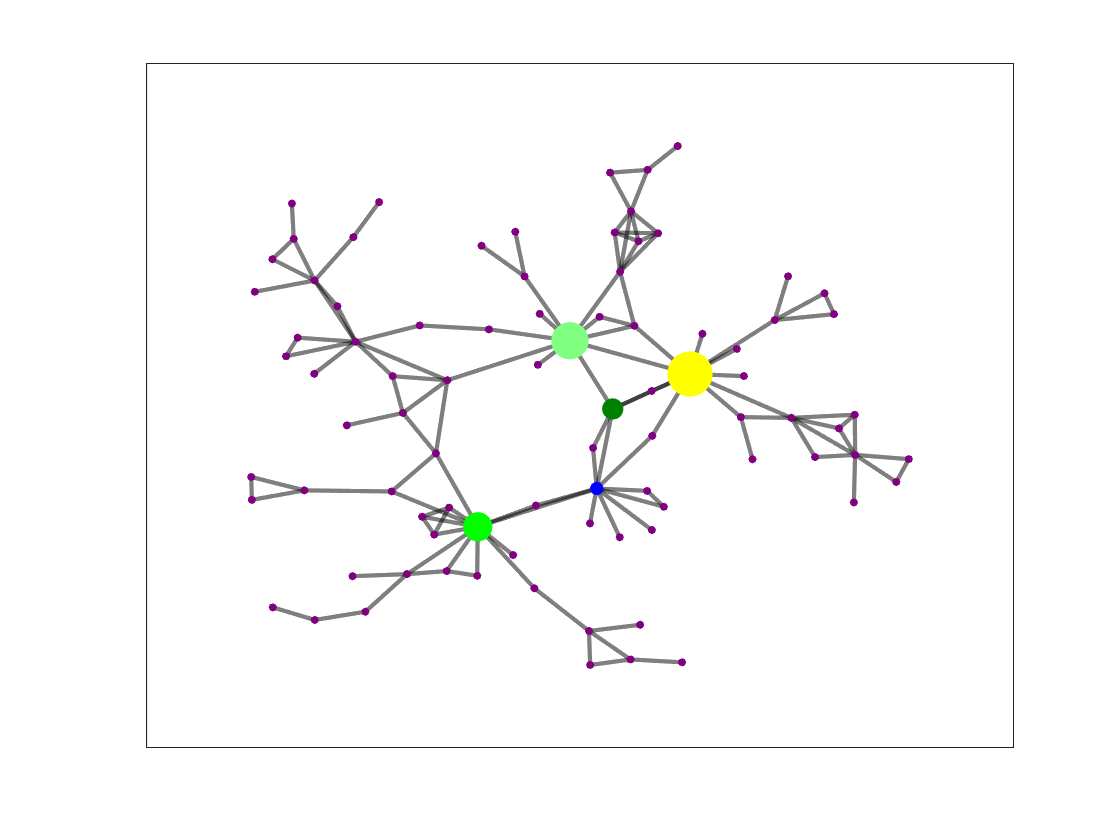}
}
\subfloat[ $(K_2,K_3)C$]{
\includegraphics[width=0.35\textwidth]{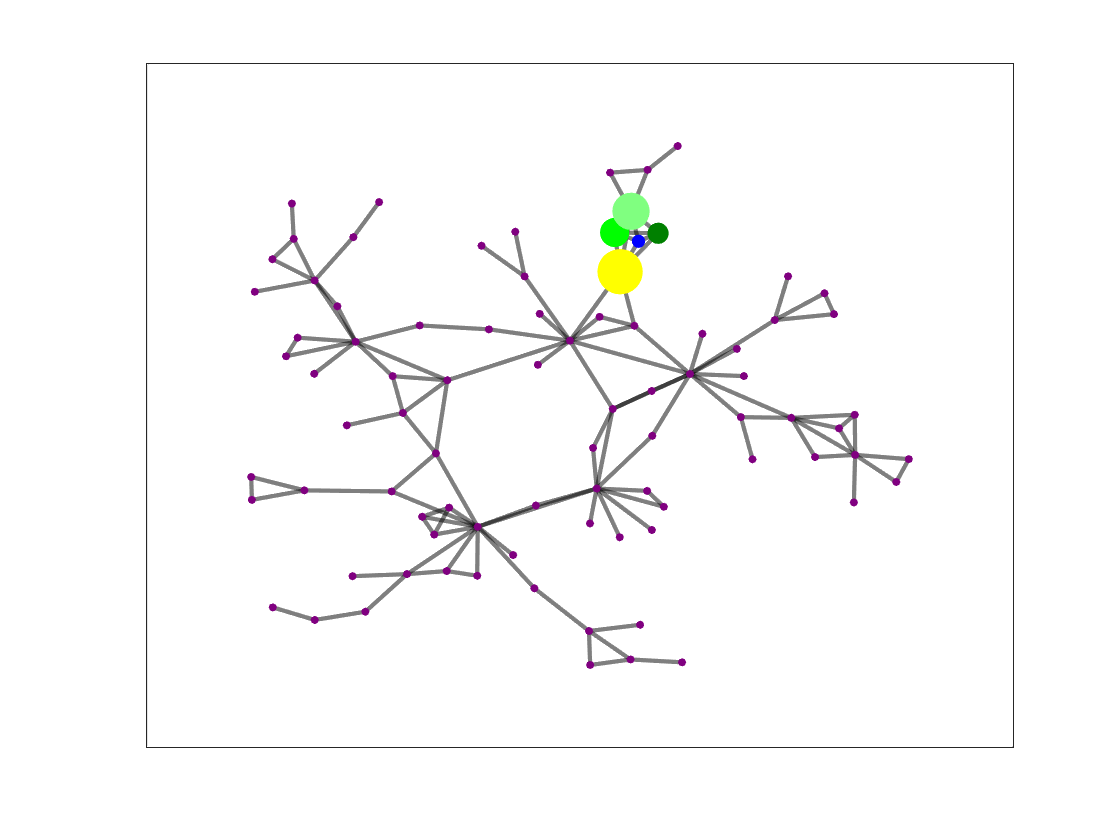}
}
\\ 
\subfloat[$BC$]{
\includegraphics[width=0.35\textwidth]{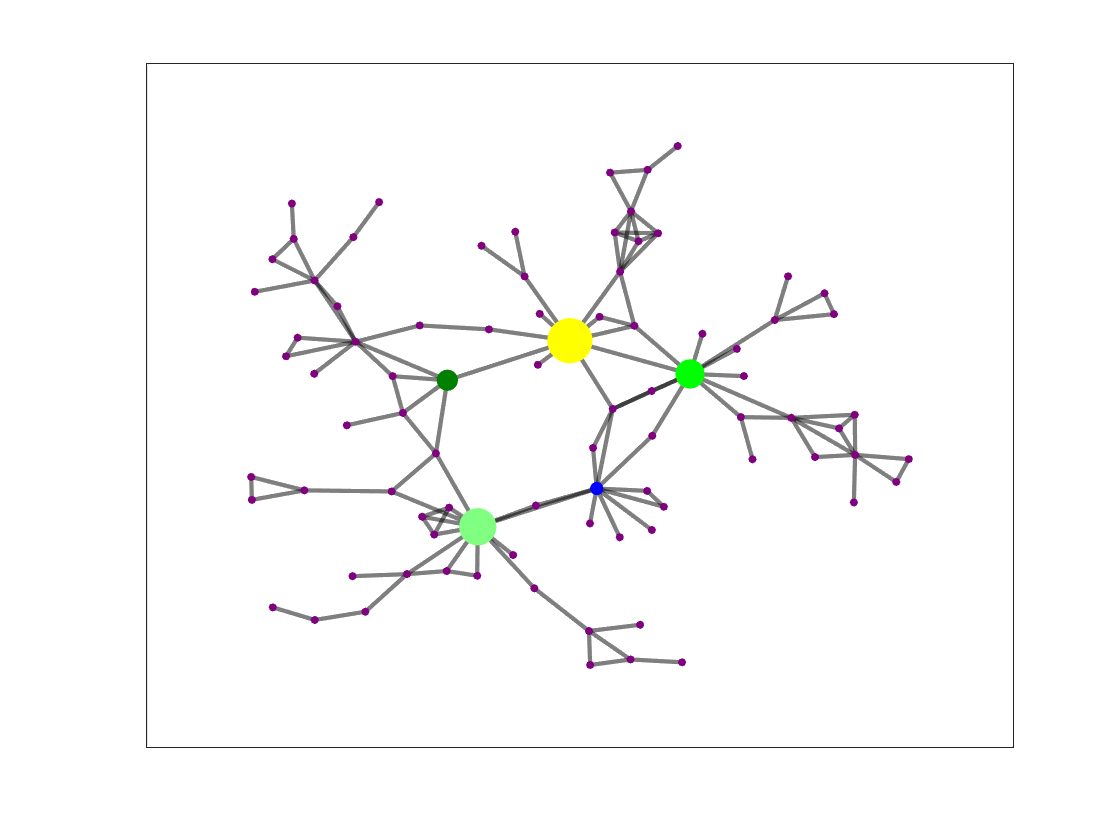}
}
\subfloat[$SC$]{
\includegraphics[width=0.35\textwidth]{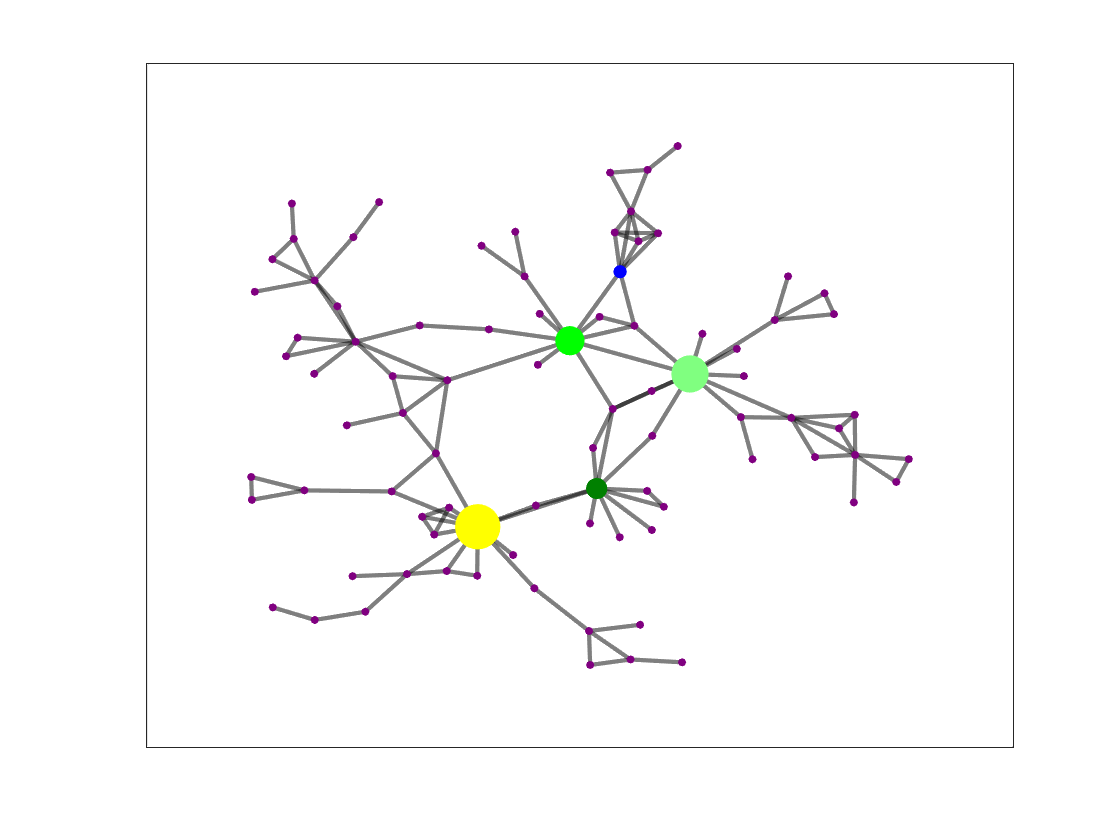}
}
\subfloat[$TS$]{
\includegraphics[width=0.35\textwidth]{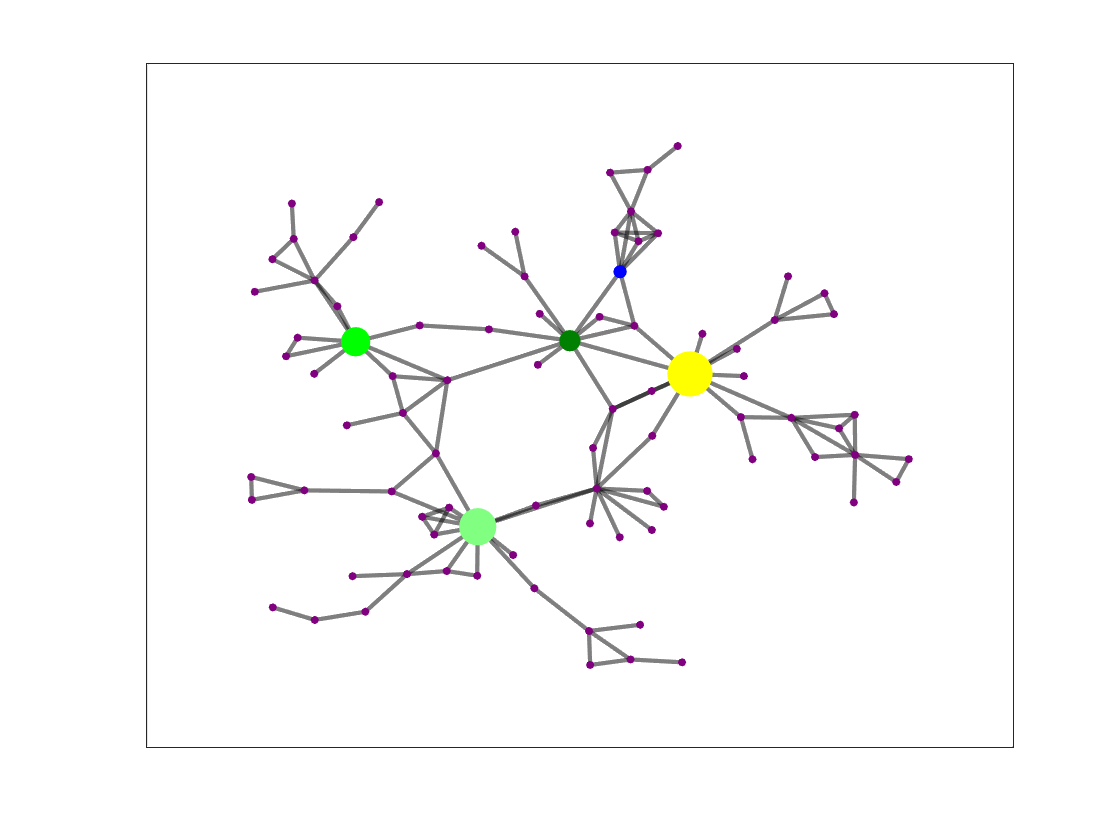}
}
\caption{The heatmap of the Sandi-Auths network under six centrality measures.}
\label{fig11}
\end{figure}

From Table \ref{table4}, we can observe that $P_2C$ and $(K_2,K_3)C$ have different rankings from other centrality measures.
Specifically, the ranking under $(K_2,K_3)C$ shows a significant difference when compare to other centrality measures.
In the top five rankings of $P_2C$ and $TS$, there are three common vertices. In the top five rankings of $(K_2,K_3)C$  and
$TS$, there is only one common vertex.
By observing Figure \ref{fig11}, it can be seen that vertices ranking higher in $P_2C$(or $(K_2,K_3)C$) are usually in more subgraph
$P_2$(or $K_3$).
However, the ranking of a vertex does not entirely depend on the number of given subgraphs contain the vertex. The centrality measures proposed in this paper are global centrality measures.

%

\subsection{The Zachary's karate club}
The Zachary's karate club network, as shown in Figure \ref{fig7}, consists of $34$ vertices and $78$ edges, represent the friendships among club members observed over two years \cite{zachary1977information}.
Due to a disagreement between the club administrator and the instructor, the club split into two smaller groups. We calculate the $P_2C$ and $(K_2,K_3)C$  for the  karate club and compare them with $EC$, $BC$ and $SC$.

\begin{figure}[H]
\centerline{\includegraphics[scale=0.24]{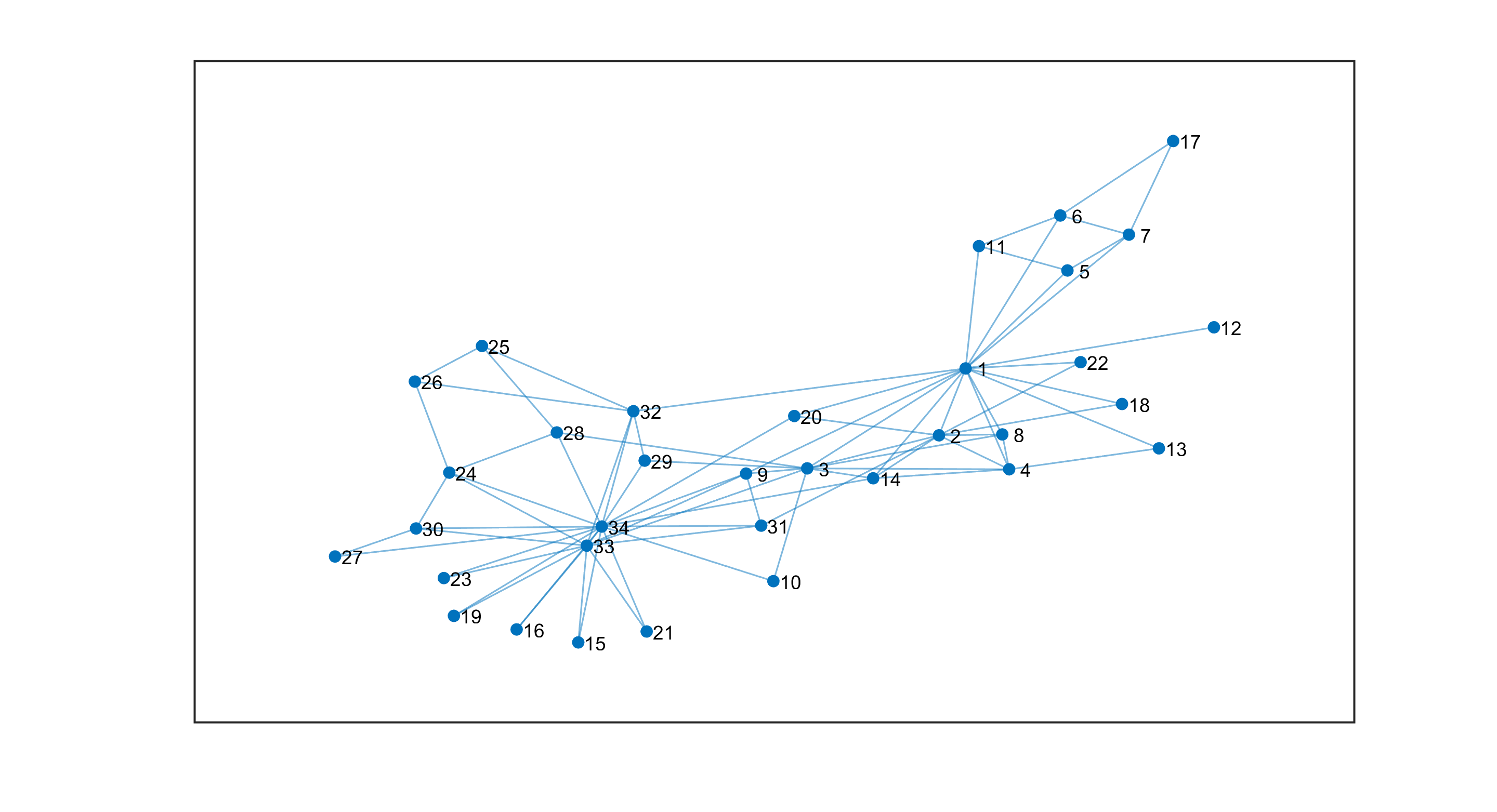}}
\caption{The Zachary's karate club graph.}
\label{fig7}
\end{figure}

In Figure \ref{fig9}, we plot the scatter plot matrix of $P_2C$, $(K_2,K_3)C$, $EC$, $BC$ and $SC$.
The diagonal elements of the scatter plot matrix are histograms of the score distributions for centrality measures, while the off-diagonal elements represent the correlations between two centrality measures.
Compared to other centrality measures, the distribution of centrality scores for vertices under $P_2C$ is more dispersed, and the differences are relatively small.
The numerical differences in centrality scores of vertices under $(K_2,K_3)C$ are relatively large.
Referring to Table \ref{table3}, we observe that $P_2C$ has a strong correlation with other centrality measures, especially with $EC$, while the correlations of $(K_2,K_3)C$ with other centrality measures are all relatively weak.

\begin{figure}[htbp]
\centering
\begin{minipage}[t]{0.45\textwidth}
\centering
\includegraphics[width=\textwidth]{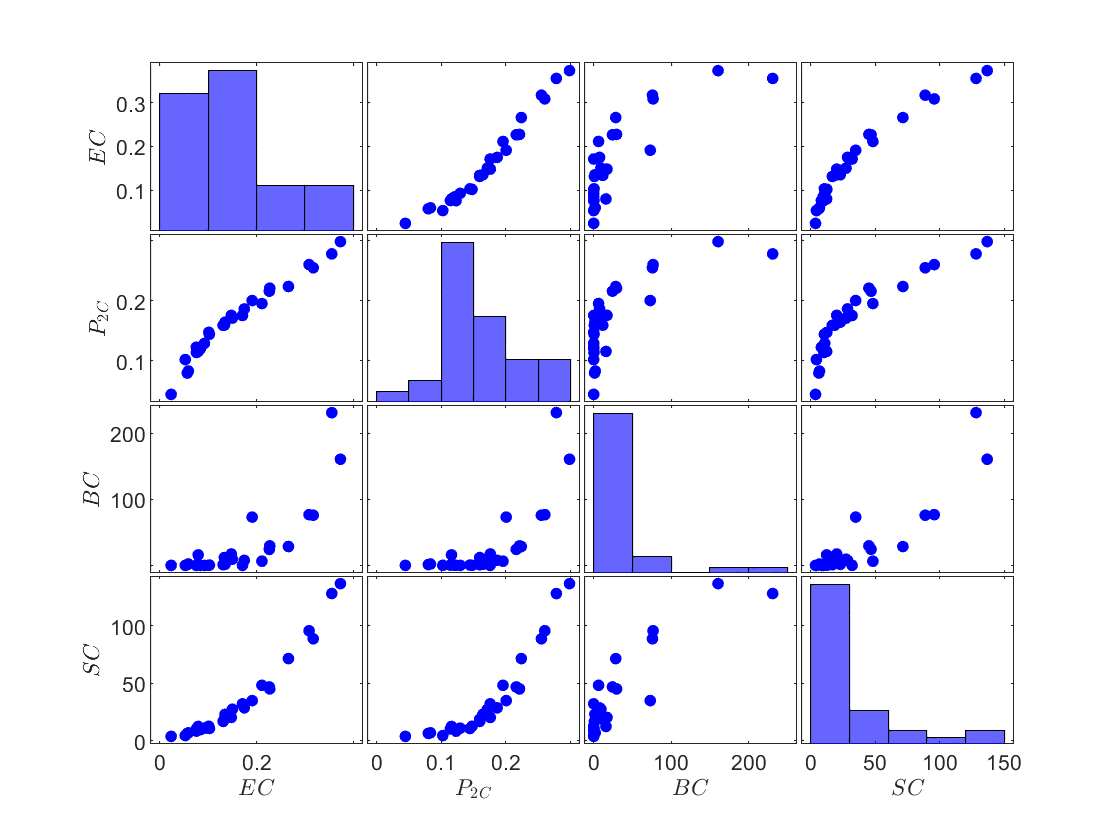}
\end{minipage}
\hspace{0.05\textwidth} 
\begin{minipage}[t]{0.45\textwidth}
\centering
\includegraphics[width=\textwidth]{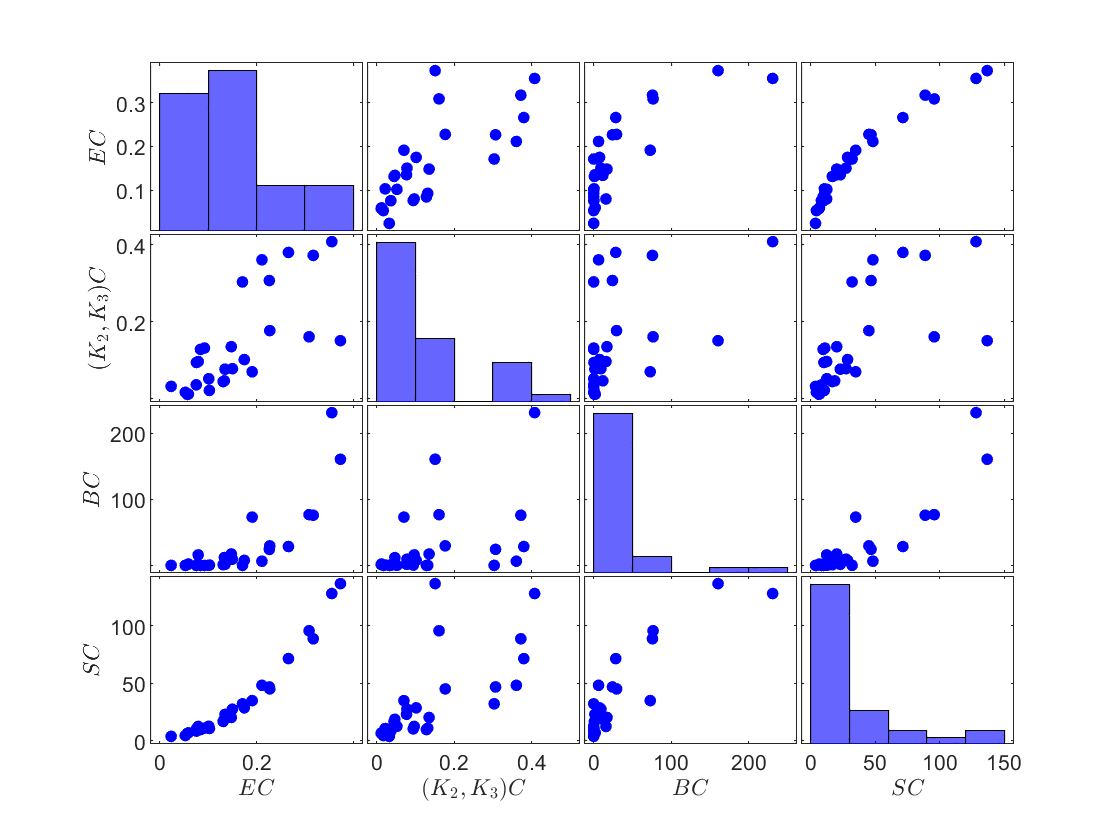}
\end{minipage}
\caption{Scatter plot matrix of the five centrality measures of the Zachary's karate club network.}
\label{fig9}
\end{figure}

\begin{table*}[h]
    \small
    \centering
    \caption{The top ten vertices in Figure 3 under five centrality measures.}
    \label{table3}
    \begin{tabular}{lccccccccc}  
    \toprule
        \textbf{Measures} & \textbf{rank1} & \textbf{rank2} & \textbf{rank3} & \textbf{rank4} & \textbf{rank5}  \\ \midrule
        $EC$        & 34 & 1 & 3 & 33 & 2  \\
        $P_{2}C$   & 34 & 1 & 33 & 3 & 2  \\ 
        $(K_2,K_3)C$   & 1 & 2 & 3 & 4 & 14   \\
        $BC$        & 1 & 34 & 33 & 3 & 32  \\
        $SC$        & 34 & 1 & 33 & 3 & 2 \\
    \toprule
        \textbf{Measures} & \textbf{rank6} & \textbf{rank7} & \textbf{rank8} & \textbf{rank9} & \textbf{rank10} \\ \midrule
        $EC$        & 9 & 14 & 4  & 32  &31  \\
        $P_{2}C$   & 9 & 14 & 32  & 4   &31  \\
        $(K_2,K_3)C$    & 8 &9 &20 &18 &22  \\
       $BC$        & 9 & 2& 14 & 20 & 6  \\
       $SC$       & 4 & 14 & 9 & 32 & 8 \\
    \bottomrule
    \end{tabular}
\end{table*}


\section{Conclusion}
The eigenvector centrality of graphs considers  the centrality of a vertex is influenced by its adjacent vertices, and uses a matrix to represent whether vertices are adjacent, i.e., whether they are on the same edge.
As research deepens, we consider that the centrality of a vertex is influenced by the big subgraphs that contain the vertex.
Naturally, the relationship between vertices and subgraphs is represented by a tensor.
In recent years, tensor has been widely used in various fields to address high-dimensional problems, and it may be an inevitable trend to use tensors to study centrality of graphs.
In this paper, for a given structural subgraph $F$ of graph $G$, we consider that the centrality of a vertex $i$ of $G$ is determined by the centrality of other vertices in all subgraphs contain $i$ and isomorphic to $F$.
We propose the $F$-subgraph eigenvector centrality and the ($K_2,F$)-subgraph eigenvector centrality of $G$.
When we choose $F$ to be $P_1$(or $K_2$), the $F$-subgraph eigenvector centrality is eigenvector centrality of graphs.
The necessary and sufficient conditions for the existence of both $P_2$-subgraph eigenvector centrality and the ($K_2,F$)-subgraph eigenvector centrality are also provided in this paper.
In regular graphs, vertices have the same eigenvector centrality scores.
The $(K_2,K_3)$-eigenvector centrality proposed in this paper can distinguish vertices in a given regular graph.
We apply the centrality measures proposed in this paper to some real-world networks. The experimental results indicate that vertices with higher centrality scores are generally  in a greater number of given subgraphs.
The centrality measures proposed in this paper are global centrality measures,
the ranking of a vertex does not entirely depend on the number of given subgraphs contain the vertex.
.

\section*{Acknowledgments}
This work is supported by the National Natural Science Foundation of China (No. 12071097, 12371344), the Natural Science Foundation for The Excellent Youth Scholars of the Heilongjiang Province (No. YQ2022A002) and the Fundamental Research Funds for the Central Universities.

\section*{References}
\bibliographystyle{plain}
\bibliography{ml0ht2}

\begin{thebibliography}{10}

\bibitem{alahakoon2011k}
Tharaka Alahakoon, Rahul Tripathi, Nicolas Kourtellis, Ramanuja Simha, and
  Adriana Iamnitchi.
\newblock K-path centrality: A new centrality measure in social networks.
\newblock In {\em Proceedings of the 4th workshop on social network systems},
  pages 1--6, 2011.

\bibitem{albert2000error}
R{\'e}ka Albert, Hawoong Jeong, and Albert~L{\'a}szl{\'o} Barab{\'a}si.
\newblock Error and attack tolerance of complex networks.
\newblock {\em Nature}, 406(6794):378--382, 2000.

\bibitem{benson2019three}
Austin~R Benson.
\newblock Three hypergraph eigenvector centralities.
\newblock {\em SIAM Journal on Mathematics of Data Science}, 1(2):293--312,
  2019.

\bibitem{benson2015tensor}
Austin~R Benson, David~F Gleich, and Jure Leskovec.
\newblock Tensor spectral clustering for partitioning higher-order network
  structures.
\newblock In {\em Proceedings of the 2015 SIAM International Conference on Data
  Mining}, pages 118--126. SIAM, 2015.

\bibitem{bihari2015eigenvector}
Anand Bihari and Manoj~Kumar Pandia.
\newblock Eigenvector centrality and its application in research professionals'
  relationship network.
\newblock In {\em 2015 International Conference on Futuristic Trends on
  Computational Analysis and Knowledge Management (ABLAZE)}, pages 510--514.
  IEEE, 2015.

\bibitem{bonacich1972factoring}
Phillip Bonacich.
\newblock Factoring and weighting approaches to status scores and clique
  identification.
\newblock {\em Journal of mathematical sociology}, 2(1):113--120, 1972.

\bibitem{bugedo2024family}
Sebasti{\'a}n Bugedo, Cristian Riveros, and Jorge Salas.
\newblock A family of centrality measures for graph data based on subgraphs.
\newblock {\em ACM Transactions on Database Systems}, 49(10):1--45, 2024.

\bibitem{chang2020hypergraph}
Jingya Chang, Yannan Chen, Liqun Qi, and Hong Yan.
\newblock Hypergraph clustering using a new laplacian tensor with applications
  in image processing.
\newblock {\em SIAM Journal on Imaging Sciences}, 13(3):1157--1178, 2020.

\bibitem{chang2008perron}
Kung-Ching Chang, Kelly Pearson, and Tan Zhang.
\newblock Perron-frobenius theorem for nonnegative tensors.
\newblock {\em Communications in Mathematical Sciences}, 6(2):507--520, 2008.

\bibitem{chen2024spectra}
Lixiang Chen, Edwin~R van Dam, and Changjiang Bu.
\newblock Spectra of power hypergraphs and signed graphs via parity-closed
  walks.
\newblock {\em Journal of Combinatorial Theory, Series A}, 207:105909, 2024.

\bibitem{chen2023c}
Yannan Chen, Antal J{\'a}kli, and Liqun Qi.
\newblock The c-eigenvalue of third order tensors and its application in
  crystals.
\newblock {\em Journal of Industrial and Management Optimization}, 19(1), 2023.

\bibitem{cooper2012spectra}
Joshua Cooper and Aaron Dutle.
\newblock Spectra of uniform hypergraphs.
\newblock {\em Linear Algebra and its applications}, 436(9):3268--3292, 2012.

\bibitem{cvetkovic1980spectra}
Dragos~M Cvetkovic, Michael Doob, and Horst Sachs.
\newblock Spectra of graphs: Theory and application.
\newblock 1980.

\bibitem{de2022communication}
Omar De~la Cruz~Cabrera, Jiafeng Jin, Silvia Noschese, and Lothar Reichel.
\newblock Communication in complex networks.
\newblock {\em Applied Numerical Mathematics}, 172:186--205, 2022.

\bibitem{estrada2005subgraph}
Ernesto Estrada and Juan~A Rodriguez-Velazquez.
\newblock Subgraph centrality in complex networks.
\newblock {\em Physical Review E}, 71(5):056103, 2005.

\bibitem{freeman1991centrality}
Linton~C Freeman, Stephen~P Borgatti, and Douglas~R White.
\newblock Centrality in valued graphs: A measure of betweenness based on
  network flow.
\newblock {\em Social networks}, 13(2):141--154, 1991.

\bibitem{friedland2013perron}
Shmuel Friedland, St{\'e}phane Gaubert, and Lixing Han.
\newblock Perron--frobenius theorem for nonnegative multilinear forms and
  extensions.
\newblock {\em Linear Algebra and its Applications}, 438(2):738--749, 2013.

\bibitem{langville2005survey}
Amy~N Langville and Carl~D Meyer.
\newblock A survey of eigenvector methods for web information retrieval.
\newblock {\em SIAM review}, 47(1):135--161, 2005.

\bibitem{lim2005singular}
Lek-Heng Lim.
\newblock Singular values and eigenvalues of tensors: a variational approach.
\newblock In {\em 1st IEEE International Workshop on Computational Advances in
  Multi-Sensor Adaptive Processing}, pages 129--132. IEEE, 2005.

\bibitem{liu2023generalization}
Chunmeng Liu and Changjiang Bu.
\newblock On a generalization of the spectral mantel's theorem.
\newblock {\em Journal of Combinatorial Optimization}, 46(2):14, 2023.

\bibitem{liu2023high}
Chunmeng Liu, Jiang Zhou, and Changjiang Bu.
\newblock The high order spectrum of a graph and its applications in graph
  colouring and clique counting.
\newblock {\em Linear and Multilinear Algebra}, 71(14):2354--2365, 2023.

\bibitem{lohmann2010eigenvector}
Gabriele Lohmann, Daniel~S Margulies, Annette Horstmann, Burkhard Pleger,
  Joeran Lepsien, Dirk Goldhahn, Haiko Schloegl, Michael Stumvoll, Arno
  Villringer, and Robert Turner.
\newblock Eigenvector centrality mapping for analyzing connectivity patterns in
  f{MRI} {D}ata of the human brain.
\newblock {\em PloS one}, 5(4):e10232, 2010.

\bibitem{meyer2023matrix}
Carl~D Meyer.
\newblock {\em Matrix Analysis and Applied Linear Algebra}.
\newblock SIAM, Philadelphia, 2023.

\bibitem{negre2018eigenvector}
Christian~FA Negre, Uriel~N Morzan, Heidi~P Hendrickson, Rhitankar Pal,
  George~P Lisi, J~Patrick Loria, Ivan Rivalta, Junming Ho, and Victor~S
  Batista.
\newblock Eigenvector centrality for characterization of protein allosteric
  pathways.
\newblock {\em Proceedings of the National Academy of Sciences},
  115(52):E12201--E12208, 2018.

\bibitem{nikabadze2016eigenvalue}
Mikhail~U Nikabadze.
\newblock Eigenvalue problems of a tensor and a tensor-block matrix (tmb) of
  any even rank with some applications in mechanics.
\newblock In {\em Generalized Continua as Models for Classical and Advanced
  Materials}, pages 279--317. Springer, 2016.

\bibitem{petersdorf1969spektrum}
M~Petersdorf and Horst Sachs.
\newblock {\"U}ber spektrum, automorphismengruppe und teiler eines graphen.
\newblock {\em Wiss. Z. Techn. Hochsch. Ilmenau}, 15:123--128, 1969.

\bibitem{qi2005eigenvalues}
Liqun Qi.
\newblock Eigenvalues of a real supersymmetric tensor.
\newblock {\em Journal of Symbolic Computation}, 40(6):1302--1324, 2005.

\bibitem{qi2017tensor}
Liqun Qi and Ziyan Luo.
\newblock {\em Tensor Analysis: Spectral Theory and Special Tensors}.
\newblock SIAM, Philadelphia, 2017.

\bibitem{qi2017eb}
Xingqin Qi, Huimin Song, Jianliang Wu, Edgar Fuller, Rong Luo, and Cun-Quan
  Zhang.
\newblock Eb\&d: A new clustering approach for signed social networks based on
  both edge-betweenness centrality and density of subgraphs.
\newblock {\em Physica A: Statistical Mechanics and its Applications},
  482:147--157, 2017.

\bibitem{rosvall2014memory}
Martin Rosvall, Alcides~V Esquivel, Andrea Lancichinetti, Jevin~D West, and
  Renaud Lambiotte.
\newblock Memory in network flows and its effects on spreading dynamics and
  community detection.
\newblock {\em Nature communications}, 5(1):4630, 2014.

\bibitem{sun2023nf}
Cheng-Wei Sun, Ting-Zhu Huang, Ting Xu, and Liang-Jian Deng.
\newblock Nf-3dlogtnn: An effective hyperspectral and multispectral image
  fusion method based on nonlocal low-fibered-rank regularization.
\newblock {\em Applied Mathematical Modelling}, 118:780--797, 2023.

\bibitem{sun2016moore}
Lizhu Sun, Baodong Zheng, Changjiang Bu, and Yimin Wei.
\newblock Moore--penrose inverse of tensors via einstein product.
\newblock {\em Linear and Multilinear Algebra}, 64(4):686--698, 2016.

\bibitem{tudisco2018node}
Francesco Tudisco, Francesca Arrigo, and Antoine Gautier.
\newblock Node and layer eigenvector centralities for multiplex networks.
\newblock {\em SIAM Journal on Applied Mathematics}, 78(2):853--876, 2018.

\bibitem{walteros2019detecting}
Jose~L Walteros, Alexander Veremyev, Panos~M Pardalos, and Eduardo~L Pasiliao.
\newblock Detecting critical node structures on graphs: A mathematical
  programming approach.
\newblock {\em Networks}, 73(1):48--88, 2019.

\bibitem{wang2018new}
Dingjie Wang and Xiufen Zou.
\newblock A new centrality measure of nodes in multilayer networks under the
  framework of tensor computation.
\newblock {\em Applied Mathematical Modelling}, 54:46--63, 2018.

\bibitem{xu2023two}
Qing Xu, Lizhu Sun, and Changjiang Bu.
\newblock The two-steps eigenvector centrality in complex networks.
\newblock {\em Chaos, Solitons and Fractals}, 173:113753, 2023.

\bibitem{yang2010further}
Yuning Yang and Qingzhi Yang.
\newblock Further results for perron--frobenius theorem for nonnegative
  tensors.
\newblock {\em SIAM Journal on Matrix Analysis and Applications},
  31(5):2517--2530, 2010.

\bibitem{zachary1977information}
Wayne~W Zachary.
\newblock An information flow model for conflict and fission in small groups.
\newblock {\em Journal of anthropological research}, 33(4):452--473, 1977.

\bibitem{zhang2020community}
Teng Zhang, Lizhu Sun, and Changjiang Bu.
\newblock Community detection via a triangle and edge combination conductance
  partitioning.
\newblock {\em Journal of Statistical Mechanics: Theory and Experiment},
  2020(7):073405, 2020.

\bibitem{zhou2013efficient}
Guanglu Zhou, Liqun Qi, and Soon-Yi Wu.
\newblock Efficient algorithms for computing the largest eigenvalue of a
  nonnegative tensor.
\newblock {\em Frontiers of mathematics in China}, 8:155--168, 2013.

\bibitem{zhou2023estrada}
Hong Zhou, Lizhu Sun, and Changjiang Bu.
\newblock Estrada index and subgraph centrality of hypergraphs via tensors.
\newblock {\em Discrete Applied Mathematics}, 341:120--129, 2023.

\end{thebibliography}
\end{spacing}
\end{document}